\documentclass[11pt,reqno]{amsart}

\newtheorem{proposition}{Proposition}[section]

\newtheorem{corollary}[proposition]{Corollary}
\newtheorem{theorem}[proposition]{Theorem}

\theoremstyle{definition}
\newtheorem{definition}[proposition]{Definition}

\newtheorem{examples}[proposition]{Examples}
\newtheorem{remark}[proposition]{Remark}

\newcommand{\thlabel}[1]{\label{th:#1}}
\newcommand{\thref}[1]{Theorem~\ref{th:#1}}

\newcommand{\colabel}[1]{\label{co:#1}}
\newcommand{\coref}[1]{Corollary~\ref{co:#1}}

\newcommand{\exlabel}[1]{\label{ex:#1}}

\newcommand{\eqlabel}[1]{\label{eq:#1}}
\newcommand{\equref}[1]{(\ref{eq:#1})}

\def\ot{\otimes}

\def\ZZ{{\mathbb Z}}

\newcommand{\Cc}{\mathcal{C}}

\def\*C{{}^*\hspace*{-1pt}{\Cc}}
\def\text#1{{\rm {\rm #1}}}

\input xy
\xyoption {all} \CompileMatrices

\usepackage{color,amssymb,graphicx}

\begin{document}

\title[Coquasitriangular structures for extensions of Hopf algebras. Applications]
{Coquasitriangular structures for extensions of Hopf algebras.
Applications}

\author{A. L. Agore}
\address{Faculty of Engineering, Vrije Universiteit Brussel, Pleinlaan 2, B-1050 Brussels, Belgium}
\email{ana.agore@vub.ac.be and ana.agore@gmail.com}



\begin{abstract}
Let $A \subseteq E$ be an extension of Hopf algebras such that
there exists a normal left $A$-module coalgebra map $\pi : E \to
A$ that splits the inclusion. We shall describe the set of all
coquasitriangular structures on the Hopf algebra $E$ in terms of
the datum $(A, E, \pi)$ as follows: first, any such extension $E$
is isomorphic to a unified product $A \ltimes H$, for some unitary
subcoalgebra $H$ of $E$ (\cite{am2}). Then, as a main theorem, we
establish a bijective correspondence between the set of all
coquasitriangular structures on an arbitrary unified product $A
\ltimes H$ and a certain set of datum $(p, \tau, u, v)$ related to
the components of the unified product. As the main application, we
derive necessary and sufficient conditions for Majid's infinite
dimensional quantum double $D_{\lambda}(A, H) = A \bowtie_{\tau}
H$ to be a coquasitriangular Hopf algebra. Several examples are
worked out in detail.
\end{abstract}
\maketitle

\section*{Introduction}

An important class of Hopf algebras is that of quasitriangular
Hopf algebras or strict quantum groups. They were introduced by
Drinfeld in \cite{Dri} as a remarkable tool for studying the
quantum Yang-Baxter equation $R^{12}R^{13}R^{23} =
R^{23}R^{13}R^{12}$. That is, if M is a representation of a
quasitriangular Hopf algebra $(H, R)$, then the canonical map
$m\ot n \mapsto \sum R^1 m \ot R^2 n$ is a solution for the
quantum Yang-Baxter equation. The dual concept, namely that of a
coquasitriangular Hopf algebra (also called braided Hopf algebras
in \cite{JB1}, \cite{JB2} or \cite{MLZ}) was first introduced by
Majid in \cite{Mj} and independently by Larson and Towber in
\cite{LT}. These are Hopf algebras $A$ endowed with a linear map
$p: A \ot A \rightarrow k$ satisfying some compatibility
conditions. There is, of course, a dual result concerning the
quantum Yang-Baxter equation: if M is a corepresentation of a
coquasitriangular Hopf algebra $(A, p)$ then the canonical map
$R_p(m \ot n) = p\bigl(m_{<1>}, n_{<1>}\bigl) m_{<0>} \ot n_{<0>}$
is a solution for the quantum Yang-Baxter equation. However, what
makes the coquasitriangular Hopf algebras so important is the fact
that the converse of the above statement is also true. Namely, by
the celebrated FRT theorem for any solution $R$ of the quantum
Yang-Baxter equation there exists a quasitriangular bialgebra $(
A(R), p)$ such that $R = R_{p}$ (\cite{book}).

Based on this background, (co)quasitriangular Hopf algebras
generated an explosion of interest and were studied for their
implications in quantum groups, the construction of invariants of
knots and $3$-manifolds, statistical mechanics, quantum mechanics
but they also became a subject of research in its own right.
Complete descriptions of the coquasitriangular structures have
been already obtained for several families of Hopf algebras, see
for instance \cite{a}, \cite{JB1}, \cite{JB2} or \cite{MLZ}.

Among the many research topics related to coquasitriangular Hopf
algebras one is of particular interest: for a given Hopf algebra
$H$, describe (if any) all coquasitriangular structures that can
be defined on $H$. We can formulate the more general problem:

\textit{Let $A \subseteq E$ be an extension of Hopf algebras. What
is the connection between the coquasitriangular structures of $A$
and those of $E$?}

Obviously, if $(E, p)$ is a coquasitriangular Hopf algebra then
$A$ is also a coquasitriangular Hopf algebra with the
coquasitriangular structure given by the restriction of $p$ to $A
\ot A$. The difficult part of the problem is the converse: if
$\sigma : A \ot A \to k$ is a coquasitriangular structure on $A$,
could it be extended to a coquasitriangular structure on $E$? In
this paper we give a complete answer to this problem in the case
when the extension $A \subseteq E$ splits in the sense of
\cite{am2}: i.e. there exists $\pi : E \to A$  a normal left
$A$-module coalgebra map such that $\pi (a) = a$, for all $a \in
A$. It was proved in \cite{am2} that an extension $A \subseteq E$
splits in the above sense if and only if $E$ is isomorphic to a
unified product between $A$ and a certain subcoalgebra $H$ of $E$.
The unified product was introduced in \cite{am1} as an answer to
the restricted extending structures problem for Hopf algebras.
Unified products characterize Hopf algebras which factorize
through a Hopf subalgebra $A$ and a subcoalgebra $H$ such that
$1\in H$. As special cases of the unified product we recover the
double cross product or the crossed product of Hopf algebras (see
Examples 1.1).

An outline of the paper is as follows. In Section 1 we recall the
construction and some basic properties of unified products. In
Section 2, the notions of generalized $(p,f)$ - left/right skew
pairing and generalized $(u,v)$ - braidings are introduced. The
main result of the paper is \thref{4.5} where a bijective
correspondence between the set of all coquasitriangular structures
$\sigma$ on the unified product $A \ltimes H$ and the set of all
quadruples $(p, \tau, u, v)$ satisfying some compatabilities is
established. All coquasitriangular structures on the unified
product are explicitly described in terms of this quadruple $(p,
\tau, u, v)$. In particular, in \coref{2.7} necessary and
sufficient conditions for a double cross product $A \bowtie H$
associated to a matched pair $(A, H, \rhd, \lhd)$ of Hopf algebras
to be a coquasitriangular Hopf algebra are given.

Let $\lambda: H \ot A \rightarrow k$ be a skew pairing between two
Hopf algebras and consider $D_{\lambda}(A, H) := A
\bowtie_{\lambda} H$ to be the \textit{generalized quantum double}
as constructed in (\cite[Example 7.2.6]{majid}). As the main
application of the results in \thref{3.1} the set of all
coquasitriangular structures on the generalized quantum double $
D_{\lambda}(A, H)$ are completely described. In particular, it is
proved that a generalized quantum double is a coquasitriangular
Hopf algebra if and only if both Hopf algebras $A$ and $H$ are
coquasitriangular. Several explicit examples are also provided.

\section{Preliminaries}
Throughout this paper k denotes an arbitrary field. Unless
specified otherwise, all algebras, coalgebras, tensor products and
homomorphisms are over $k$. For a coalgebra $C$, we use Sweedler's
$\Sigma$-notation: $\Delta(c) = c_{(1)}\ot c_{(2)}$,
$(I\ot\Delta)\Delta(c) = c_{(1)}\ot c_{(2)}\ot c_{(3)}$, etc with
summation understood. For a $k$-linear map $f: H \ot H \to A$ we
denote $f(g, \, h) = f (g\ot h)$.\\
Recall from \cite{DT2} that if $A$ and $H$ are two Hopf algebras
and $\lambda: A \otimes H \rightarrow k$ is a $k$-linear map which
fulfills the compatibilities:
\begin{enumerate}
\item[(BR1)] $\lambda(xy , z) = \lambda(x , z_{(1)}) \lambda(y ,
z_{(2)})$\\
\item[(BR2)] $\lambda(1 , z) = \varepsilon(z)$\\
\item[(BR3)] $\lambda(x , lz) = \lambda(x_{(1)} , z)
\lambda(x_{(2)}
, l)$\\
\item[(BR4)] $\lambda(y , 1) = \varepsilon(y)$\\
\end{enumerate}
for all $x$, $y \in A$, $l$, $z \in H$, then $\lambda$ is called
\textit{skew pairing} on $(A, H)$. Notice that a skew pairing
$\lambda$ is convolution invertible with $\lambda^{-1} = \lambda
\circ (S \ot Id)$. Also, by a straightforward computation it can
be seen that if $\lambda$ is a skew pairing on $(A, H)$ then
$\lambda \circ (S \ot Id) \circ \nu$ is also a skew pairing on
$(H, A)$ where $\nu$ is the flip map.

Moreover, recall from \cite{LT} that a Hopf algebra $H$ is called
\textit{coquasitriangular} or \textit{braided} if there exists a
linear map $p: H \otimes H \rightarrow k$ such that relations
$(BR1)-(BR4)$ are fulfilled and
\begin{enumerate}
\item[(BR5)] $p(x_{(1)} , y_{(1)})x_{(2)}y_{(2)} = y_{(1)}x_{(1)}
p(x_{(2)} , y_{(2)})$ \end{enumerate} holds for all $x$, $y$, $z
\in H$.

\subsection*{Unified products} We recall from \cite{am1} the construction of the unified product.
An \textit{extending datum} of a bialgebra $A$ is a system
$\Omega(A) = \bigl(H, \triangleleft, \, \triangleright, \, f
\bigl)$, where $H = \bigl( H, \Delta_{H}, \varepsilon_{H},\\
1_{H}, \cdot \bigl)$ is a $k$-module such that $\bigl( H,
\Delta_{H}, \varepsilon_{H}\bigl)$ is a coalgebra, $\bigl( H,
1_{H}, \cdot \bigl)$ is an unitary not necessarily associative
$k$-algebra, the $k$-linear maps $\triangleleft : H \otimes A
\rightarrow H$, $\triangleright: H \otimes A \rightarrow A$, $f: H
\otimes H \rightarrow A$ are coalgebra maps such that the
following normalization conditions hold:
\begin{equation}\eqlabel{2}
\quad h \triangleright 1_{A} = \varepsilon_{H}(h)1_{A}, \quad
1_{H} \triangleright a = a, \quad 1_{H} \triangleleft a =
\varepsilon_{A}(a)1_{H}, \quad h\triangleleft 1_{A} = h
\end{equation}
\begin{equation}\eqlabel{3}
\Delta_{H}(1_{H}) = 1_{H} \otimes 1_{H}, \qquad f(h, 1_{H}) =
f(1_{H}, h) = \varepsilon_{H}(h)1_{A}
\end{equation}
for all $h \in H$, $a \in A$.

Let $\Omega(A) = \bigl(H, \triangleleft, \, \triangleright, \, f
\bigl)$ be an extending datum of $A$. We denote by $A
\ltimes_{\Omega(A)} H = A \ltimes H$ the $k$-module $A \otimes H$
together with the multiplication:
\begin{equation}\eqlabel{10}
(a \ltimes h)\bullet(c \ltimes g) := a(h_{(1)}\triangleright
c_{(1)})f\bigl(h_{(2)}\triangleleft c_{(2)}, \, g_{(1)}\bigl) \,
\ltimes \, (h_{(3)}\triangleleft c_{(3)}) \cdot g_{(2)}
\end{equation}
for all $a, c \in A$ and $h, g \in H$, where we denoted $a \otimes
h \in A \otimes H$ by $a \ltimes h$. The object $A \ltimes H$ is
called \textit{the unified product of $A$ and $\Omega(A)$} if $A
\ltimes H$ is a bialgebra with the multiplication given by
\equref{10}, the unit $1_{A} \ltimes 1_{H}$ and the coalgebra
structure given by the tensor product of coalgebras, i.e.:
\begin{eqnarray}
\Delta_{A \ltimes H} (a \ltimes h) &{=}& a_{(1)} \ltimes h_{(1)}
\otimes a_{(2)}
\ltimes h_{(2)}\eqlabel{11}\\
\varepsilon_{A \ltimes H} (a \ltimes h) &{=}&
\varepsilon_{A}(a)\varepsilon_{H}(h)\eqlabel{12}
\end{eqnarray}
for all $h \in H$, $a \in A$. We have proved in \cite[Theorem
2.4]{am1} that $A \ltimes H$ is an unified product if and only if
$\Delta_{H} : H \to H\otimes H$ and $\varepsilon_{H} : H \to k$
are $k$-algebra maps, $(H, \lhd)$ is a right $A$-module structure
and the following compatibilities hold:
\begin{enumerate}
\item[(BE1)] $(g\cdot h)\cdot l = \bigl(g \triangleleft
f(h_{(1)}, \, l_{(1)})\bigl)\cdot (h_{(2)}\cdot l_{(2)})$\\
\item[(BE2)] $g \triangleright (ab) = (g_{(1)} \triangleright
a_{(1)})[(g_{(2)}\triangleleft a_{(2)})\triangleright b]$\\
\item[(BE3)] $(g\cdot h) \triangleleft a = [g \triangleleft
(h_{(1)}
\triangleright a_{(1)})] \cdot (h_{(2)} \triangleleft a_{(2)})$\\
\item[(BE4)] $[g_{(1)} \triangleright (h_{(1)} \triangleright
a_{(1)})]f\Bigl(g_{(2)} \triangleleft (h_{(2)} \triangleright
a_{(2)}), \, h_{(3)} \triangleleft a_{(3)}\Bigl) =
f(g_{(1)}, \, h_{(1)})[(g_{(2)} \cdot h_{(2)}) \triangleright a]$\\
\item[(BE5)] $\Bigl(g_{(1)} \triangleright f(h_{(1)}, \,
l_{(1)})\Bigl) f\Bigl(g_{(2)} \triangleleft f(h_{(2)}, \,
l_{(2)}), \, h_{(3)} \cdot l_{(3)}\Bigl) =
f(g_{(1)}, \, h_{(1)})f(g_{(2)} \cdot h_{(2)}, \, l)$\\
\item[(BE6)] $g_{(1)} \triangleleft a_{(1)} \otimes g_{(2)}
\triangleright a_{(2)} = g_{(2)} \triangleleft a_{(2)} \otimes
g_{(1)} \triangleright a_{(1)}$\\
\item[(BE7)] $g_{(1)} \cdot h_{(1)} \otimes f(g_{(2)}, \, h_{(2)})
= g_{(2)} \cdot h_{(2)} \otimes f(g_{(1)}, \, h_{(1)})$
\end{enumerate}
for all $g$, $h$, $l \in H$ and $a$, $b \in A$. In this case
$\Omega(A) = (H, \triangleleft, \, \triangleright, \, f)$ is
called a \emph{bialgebra extending structure} of $A$. A bialgebra
extending structure $\Omega(A) = (H, \triangleleft,
\triangleright, f)$ is called a \emph{Hopf algebra extending
structure} of A if $A \ltimes H$ has an antipode. If $A$ is a Hopf
algebra with an antipode $S_A$ and $H$ has an antipode $S_H$, then
the unified product $A \ltimes H$ has an antipode given by:
$$
S(a \ltimes g) := \Bigl(S_{A}[f\bigl(S_{H}(g_{(2)}), \,
g_{(3)}\bigl)] \ltimes S_{H}(g_{(1)})\Bigl) \bullet \bigl(S_{A}(a)
\ltimes 1_{H}\bigl)
$$
for all $a\in A$ and $g\in H$ (\cite[Proposition 2.8]{am1}).\\
In \cite{am2} it was proved that a Hopf algebra $E$ is isomorphic
to a unified product $A \ltimes H$ if and only if there exists a
morphism of Hopf algebras $i: A \rightarrow E$ which has a
retraction $\pi: E \to A$ that is a normal (\cite[Definition
2.1]{am2}) left $A$-module coalgebra morphism.

\begin{examples}\exlabel{3exemple}
1. Let $A$ be a bialgebra and $\Omega(A) = \bigl(H, \triangleleft,
\triangleright, f \bigl)$ an extending datum of $A$ such that the
cocycle $f$ is trivial, that is $f (g, \, h) = \varepsilon_H (g)
\varepsilon_H (h) 1_A$, for all $g$, $h\in H$.

Then $\Omega(A) = \bigl(H, \triangleleft, \triangleright, f
\bigl)$ is a bialgebra extending structure of $A$ if and only if
$H$ is a bialgebra and $(A, H, \triangleleft, \triangleright)$ is
a matched pair of bialgebras in the sense of \cite[Definition
7.2.1]{majid}. In this case, the associated unified product
$A\ltimes H = A\bowtie H$ is the double cross product of
bialgebras in Majid's terminology (also called bicrossed product
of bialgebras in \cite{K}). Perhaps the most famous example of a
double cross product is the generalized quantum double (see
Section 3 below). If $H$ is a finite dimensional Hopf algebra then
the generalized quantum double coincides with the celebrated
quantum double $D(H) = H^{*op} \bowtie H$ which is a double cross
product by the mutual coadjoint actions:
\begin{eqnarray*}
h \triangleright \alpha = \alpha_{(2)} \langle h , S(\alpha_{(1)})
 \alpha_{(3)} \rangle \quad , \quad
h \triangleleft \alpha = h_{(2)} \langle \alpha , S(h_{(1)})
h_{(3)} \rangle
\end{eqnarray*}
for all $h \in H$ and $\alpha \in H^{*}$.

2. Let $A$ be a bialgebra and $\Omega(A) = \bigl(H, \triangleleft,
\, \triangleright, \, f \bigl)$ an extending datum of $A$ such
that the action $\lhd$ is trivial, that is $h \lhd a =
\varepsilon_A (a) h$, for all $h\in H$ and $a\in A$. In this case,
the associated unified product $A\ltimes H = A
\#_{f}^{\triangleright} \, H$ is called the crossed product of
Hopf algebras. For more details on crossed products of Hopf
algebras we refer to \cite{a}.
\end{examples}

\section{Coquasitriangular structures on the unified products}

In this section we describe the coquasitriangular or braided
structures on the unified product. In other words, we determine
all braided structures that can be defined on the monoidal
category of $A \ltimes H$ - comodules. First we introduce some new
definitions as natural generalizations for the concepts of
braiding and skew pairing.

\begin{definition}
Let $A$ be a Hopf algebra, $H = \bigl( H, \Delta_{H},
\varepsilon_{H}, 1_{H}, \cdot \bigl)$ a $k$-module such that
$\bigl( H, \Delta_{H}, \varepsilon_{H}\bigl)$ is a coalgebra,
$\bigl( H, 1_{H}, \cdot \bigl)$ is an unitary not necessarily
associative $k$-algebra, $f: H \otimes H \rightarrow A$ a
coalgebra map and $p: A\otimes A \rightarrow k$ a braiding on $A$.
A linear map $u : A \otimes H \rightarrow k$ is called
\textit{generalized (p,f) - right skew pairing on $(A, H)$} if the
following compatibilities are fulfilled for any $a$, $b \in A$,
$g$, $t \in H$:
\begin{enumerate}
\item[(RS1)] $u(ab , t) = u(a , t_{(1)}) u(b , t_{(2)})$\\
\item[(RS2)] $u(1 , h) = \varepsilon(h)$\\
\item[(RS3)] $u(a_{(1)} , g_{(2)} \cdot t_{(2)}) p\bigl(a_{(2)} ,
f(g_{(1)} , t_{(1)})\bigl) = u(a_{(1)} , t) u(a_{(2)} , g)$\\
\item[(RS4)] $u(a , 1) = \varepsilon(a)$
\end{enumerate}
\end{definition}

\begin{definition}
Let $A$ be a Hopf algebra, $H = \bigl( H, \Delta_{H},
\varepsilon_{H}, 1_{H}, \cdot \bigl)$ a $k$-module such that
$\bigl( H, \Delta_{H}, \varepsilon_{H}\bigl)$ is a coalgebra,
$\bigl( H, 1_{H}, \cdot \bigl)$ is an unitary not necessarily
associative $k$-algebra, $f: H \otimes H \rightarrow A$ a
coalgebra map and $p: A\otimes A \rightarrow k$ a braiding on $A$.
A linear map $v: H \otimes A \rightarrow k$ is called
\textit{generalized (p,f) - left skew pairing on $(H, A)$} if the
following compatibilities are fulfilled for any $b$, $c \in A$,
$h$, $g \in H$:
\begin{enumerate}
\item[(LS1)] $p\bigl(f(h_{(1)} , g_{(1)}) , c_{(1)}\bigl)
v(h_{(2)} \cdot g_{(2)} , c_{(2)}) = v(h , c_{(1)}) v(g , c_{(2)})$\\
\item[(LS2)] $v(h , 1) = \varepsilon(h)$\\
\item[(LS3)] $v(h , bc) = v(h_{(1)} , c) v(h_{(2)} , b)$\\
\item[(LS4)] $v(1 , a) = \varepsilon(a)$
\end{enumerate}
\end{definition}

\begin{remark}
If $H$ is a bialgebra and $f = \varepsilon \ot \varepsilon$ is the
trivial cocycle then the notion of generalized (p,f) - left/right
skew pairing on $(A, H)$ coincides with the notion of skew pairing
on $(A,H)$.
\end{remark}

\begin{definition}
Let $A$ be a Hopf algebra, $H = \bigl( H, \Delta_{H},
\varepsilon_{H}, 1_{H}, \cdot \bigl)$ a $k$-module such that
$\bigl( H, \Delta_{H}, \varepsilon_{H}\bigl)$ is a coalgebra,
$\bigl( H, 1_{H}, \cdot \bigl)$ is an unitary not necessarily
associative $k$-algebra, $f: H \otimes H \rightarrow A$ a
coalgebra map and $p: A\otimes A \rightarrow k$ a braiding on $A$,
$u: A \otimes H \rightarrow k$ a generalized $(p , f)$ - right skew
pairing and $v: H \otimes A \rightarrow k$ a generalized $(p , f)$
- left skew pairing. A linear map $\tau: H \otimes H \rightarrow
k$ is called a \textit{generalized $(u,v)$ - skew braiding on $H$}
if the following compatibilities are fulfilled for all $h$, $g$,
$t \in H$:
\begin{enumerate}
\item[(SBR1)] $u\bigl( f(h_{(1)} , g_{(1)}) , t_{(1)}\bigl)
\tau(h_{(2)} \cdot g_{(2)} , t_{(2)}) = \tau(h , t_{(1)}) \tau(g ,
t_{(2)})$\\
\item[(SBR2)] $\tau(1 , g) = \varepsilon(g)$\\
\item[(SBR3)] $\tau(h_{(1)} , g_{(2)} \cdot t_{(2)})
v\bigl(h_{(2)} , f(g_{(1)} , t_{(1)}) \bigl) = \tau(h_{(1)} , t)
\tau(h_{(2)} ,
g)$\\
\item[(SBR4)] $\tau(g , 1) = \varepsilon(g)$\\
\item[(SBR5)] $\tau(h_{(1)} , g_{(1)})h_{(2)} \cdot g_{(2)} =
g_{(1)} \cdot h_{(1)} \tau(h_{(2)} , g_{(2)})$
\end{enumerate}
\end{definition}

\begin{remark}
If $H$ is a bialgebra and $f = \varepsilon \ot \varepsilon$ is the
trivial cocycle then the notion of generalized $(u,v)$ - skew
braiding on $H$ coincides with the notion of coquasitriangular
structure (or braiding) on $H$.
\end{remark}

\begin{theorem}\thlabel{4.5}
Let $A$ be a Hopf algebra and $\Omega(A) = (H, \triangleleft,
\triangleright, f)$ a Hopf algebra extending structure of $A$.
There is a bijective correspondence between:

$(i)$ The set of all coquasitriangular structures $\sigma$ on the
unified product $A \ltimes H$;

$(ii)$ The set of all quadruples $(p, \tau, u, v)$ where $p: A
\otimes A \rightarrow k$, $\tau: H \otimes H \rightarrow k$, $u: A
\otimes H \rightarrow k$, $v: H \otimes A \rightarrow k$ are
linear maps such that $(A, p)$ is a coquasitriangular Hopf
algebra, $u$ is a generalized $(p,f)$ - right skew pairing, $v$ is
a generalized $(p,f)$ - left skew pairing, $(H, \tau)$ is a
generalized $(u, v)$ - skew braiding and the following
compatibilities are fulfilled:
\begin{eqnarray}
\eqlabel{4.1} v(h_{(1)} , b_{(1)}) (h_{(2)} \triangleright
b_{(2)}) \otimes (h_{(3)} \triangleleft b_{(3)}) &{=}& b_{(1)}
\otimes h_{(1)}v(h_{(2)} ,
b_{(2)})\\
\eqlabel{4.2} (g_{(1)} \triangleright a_{(1)}) \otimes (g_{(2)}
\triangleleft a_{(2)}) u(a_{(3)} , g_{(3)}) &{=}& u(a_{(1)} ,
g_{(1)})a_{(2)} \otimes
g_{(2)}\\
\eqlabel{4.3} \tau(h_{(1)} , g_{(1)}) f(h_{(2)} , g_{(2)}) &{=}&
f(g_{(1)} , h_{(1)}) \tau(h_{(2)} , g_{(2)})\\
\eqlabel{4.4} u(a_{(1)} , g_{(2)} \triangleleft c_{(2)}) p(a_{(2)}
, g_{(1)}
\triangleright c_{(1)}) &{=}& p(a_{(1)} , c) u(a_{(2)} , g)\\
\eqlabel{4.5} \tau(h_{(1)} , g_{(2)} \triangleleft c_{(2)})
v(h_{(2)} , g_{(1)}
\triangleright c_{(1)}) &{=}& v(h_{(1)} , c) \tau(h_{(2)} , g)\\
\eqlabel{4.6} p(h_{(1)} \triangleright b_{(1)}, c_{(1)}) v(h_{(2)}
\triangleleft b_{(2)} ,
c_{(2)}) &{=}& v(h , c_{(1)}) p(b , c_{(2)})\\
\eqlabel{4.7} u(h_{(1)} \triangleright b_{(1)}, t_{(1)})
\tau(h_{(2)} \triangleleft b_{(2)} , t_{(2)}) &{=}& \tau(h ,
t_{(1)}) u(b , t_{(2)})
\end{eqnarray}
Under the above bijection the coquasitriangular structure $\sigma:
(A \ltimes H) \otimes (A \ltimes H) \rightarrow k$ corresponding
to $(p, \tau, u, v)$ is given by:
\begin{equation}\eqlabel{39}
\sigma(a \ltimes h, b \ltimes g) = u(a_{(1)} , g_{(1)}) p(a_{(2)}
, b_{(1)}) \tau(h_{(1)} , g_{(2)}) v(h_{(2)} , b_{(2)})
\end{equation}
for all $a, b, c \in A$ and $h, g, t \in H$.
\end{theorem}
\begin{proof}
Suppose first that $(A \ltimes H, \sigma)$ is a coquasitriangular
Hopf algebra. We define the following linear maps:
\begin{eqnarray*}
p: A \otimes A \rightarrow k, \qquad p(a, b) &=& \sigma(a \otimes
1
, b \otimes 1)\\
\tau: H \otimes H \rightarrow k, \qquad \tau(h, g) &=& \sigma(1
\otimes h , 1 \otimes g)\\
u: A \otimes H \rightarrow k, \qquad u(a , h) &=& \sigma(a \otimes
1
, 1 \otimes h)\\
v: H \otimes A \rightarrow k, \qquad v(h, a) &=& \sigma(1 \otimes
h, a \otimes 1)
\end{eqnarray*}
Before going into the proof we collect here some compatibilities
satisfied by the maps defined above which will be useful in the
sequel. The following are just easy consequences of the fact that
$\sigma$ is a coquasitriangular structure on $A \ltimes \, H$ and,
hence, it satisfies the normalizing relations (BR2) and (BR4):
\begin{eqnarray}
\eqlabel{a} p(1 , b) = \varepsilon(b) = p(b , 1)\\
\eqlabel{b} \tau(1 , h) = \varepsilon(h) = \tau(h , 1)\\
\eqlabel{c} u(1 , h) = \varepsilon(h) , \quad u(a , 1) =
\varepsilon(a)\\
\eqlabel{d} v(1 , a) = \varepsilon(a) , \quad v(h , 1) =
\varepsilon(h)
\end{eqnarray}
Remark that from relation \equref{b} it follows that $\tau$
fulfills (SBR2) and (SBR4) while from relation \equref{d} we can
derive that $v$ fulfills (LS2) and (LS4).

First we prove that relation \equref{39} indeed holds:
\begin{eqnarray*}
&&\sigma(a \# h,\ b \# g) = \sigma \bigl((a \# 1)(1 \# h), (b \# 1)(1 \# g)\bigl) = \\
&\stackrel{(BR1)}{=}& \sigma \bigl(a \# 1, (b_{(1)} \# 1)(1 \#
g_{(1)})\bigl) \sigma \bigl((1 \# h), (b_{(2)} \# 1)(1 \#
g_{(2)})\bigl)\\
&\stackrel{(BR3)}{=}& \sigma (a_{(1)} \# 1 , 1 \# g_{(1)}) \sigma
(a_{(2)} \# 1 , b_{(1)} \# 1) \sigma (1 \# h_{(1)} , 1 \# g_{(2)})
\sigma \bigl(1 \# h_{(2)} , b_{(2)} \# 1)\\
&{=}&u(a_{(1)} , g_{(1)}) p(a_{(2)} , b_{(1)}) \tau(h_{(1)} ,
g_{(2)}) v(h_{(2)} , b_{(2)})
\end{eqnarray*}
Next we prove that $(A, p)$ is a coquasitriangular Hopf algebra,
$u$ is a generalized $(p,f)$ - right skew pairing on $(H, A)$, $v$
is a generalized $(p,f)$ - left skew pairing on $(A, H)$ and
$\tau$ is a generalized $(u, v)$ - skew braiding on $H$. Having in
mind that $(A \ltimes H, \sigma)$ is a coquasitriangular Hopf
algebra it is straightforward to see that $(A,p)$ is a
coquasitriangular Hopf algebra by considering $x = a \# 1$, $y = b
\# 1$ and $z = c \# 1$ in (BR1) $-$ (BR5).\\
Since $ \sigma$ satisfies (BR1), then for all $a$, $b$, $c \in A$
and $h$, $g$, $t \in H$ we have:
$$
\sigma \bigl(a(g_{(1)} \triangleright b_{(1)}) f(h_{(2)}
\triangleleft b_{(2)} , h_{(1)}) \# (g_{(3)} \triangleleft
b_{(3)}) \cdot h_{(2)} , c \# t \bigl) =
$$
\begin{eqnarray}\eqlabel{BR1}
= \sigma ( a \otimes g, c_{(1)} \otimes t_{(1)}) \sigma(b \otimes
h, c_{(2)} \otimes t_{(2)})
\end{eqnarray}
Moreover, since $\sigma$ also fulfills (BR3) we have:
$$
\sigma \bigl(a \# h , b(g_{(1)} \triangleright c_{(1)}) f(g_{(2)}
\triangleleft c_{(2)} , t_{(1)}) \# (g_{(3)} \triangleleft
c_{(3)}) \cdot t_{(2)} \bigl)  =
$$
\begin{eqnarray}\eqlabel{BR3}
= \sigma ( a_{(1)} \otimes h_{(1)} , c \otimes t) \sigma(a_{(2)}
\otimes h_{(2)}, b \otimes g)
\end{eqnarray}
Furthermore, by (BR5) we have:
$$\sigma(a_{(1)} \otimes h_{(1)} , b_{(1)} \otimes
g_{(1)}) a_{(2)}(h_{(2)} \triangleright b_{(2)}) f(h_{(3)}
\triangleleft b_{(3)}, g_{(2)}) \otimes (h_{(4)} \triangleleft
b_{(4)}) \cdot g_{(3)}$$
\begin{eqnarray}
\eqlabel{BR5} = b_{(1)} (g_{(1)} \triangleright a_{(1)}) f(g_{(2)}
\triangleleft a_{(2)} , h_{(1)}) \otimes (g_{(3)} \triangleleft
a_{(3)}) \cdot h_{(2)} \sigma(a_{(4)} \otimes h_{(3)} , b_{(2)}
\otimes g_{(4)})
\end{eqnarray}
By considering $h = g = 1$ and $c = 1$ in \equref{BR1} we get
relation (RS1). If we let $b = c = 1$ and $h = 1$ in \equref{BR3}
yields:
$$
u(a_{(1)} , g_{(3)} \cdot t_{(3)}) p\bigl(a_{(2)} , f(g_{(1)} ,
t_{(1)})\bigl) \tau(1 , g_{(4)}t_{(4)}) v\bigl(1 , f(g_{(2)} ,
t_{(2)})\bigl) = u(a_{(1)} , t) u(a_{(2)} , g)
$$
Now using relations \equref{b} and \equref{d} we get (RS3). Hence
we proved that $u$ is a generalized $(p,f)$ - right skew pairing
on $(H, A)$. Considering $a = b = 1$ and $t = 1$ in \equref{BR1}
yields:
\begin{eqnarray*}
u \bigl(f(h_{(1)} , g_{(1)}) , 1\bigl) p \bigl(f(h_{(2)} ,
g_{(2)}) , c_{(1)}\bigl) \tau(h_{(3)} \cdot g_{(3)} , 1) v(h_{(4)}
\cdot g_{(4)} , c_{(2)}) = v(h , c_{(1)}) v(g , c_{(2)})
\end{eqnarray*}
Using \equref{b} and \equref{c} we get that (LS1) holds for $v$.
Moreover from \equref{BR3} applied to $g = t = 1$ and $a = 1$ we
get that (LS3) also holds for $v$ and we proved that $v$ is indeed
a generalized $(p,f)$ - left skew pairing on $(A, H)$. Next we
apply \equref{BR1} for $a = b = c = 1$:
\begin{eqnarray*}
u \bigl(f(h_{(1)} , g_{(1)}) , t_{(1)}\bigl) p \bigl(f(h_{(2)} ,
g_{(2)}) , 1\bigl) \tau(h_{(3)} \cdot g_{(3)} , t_{(2)}) v(h_{(4)}
\cdot g_{(4)} , 1) = \tau(h , t_{(1)}) \tau(g , t_{(2)})
\end{eqnarray*}
Using \equref{a} and \equref{d} we obtain (SBR1). Now \equref{BR3}
applied for $a = b = c = 1$ yields:
\begin{eqnarray*}
u(1 , g_{(3)} \cdot t_{(3)}) p\bigl(1 , f(g_{(1)} , t_{(1)})\bigl)
v \bigl(h_{(2)} , f(g_{(2)} , t_{(2)})\bigl) \tau(h_{(1)} ,
g_{(4)} \cdot t_{(4)}) = \tau(h_{(1)} , t) \tau(h_{(2)} , g)
\end{eqnarray*}
From \equref{a} and \equref{c} we obtain that (SBR3) holds for
$\tau$. Considering $a = b = 1$ in \equref{BR5} we get:
\begin{eqnarray*}
\tau(h_{(1)} , g_{(1)})f(h_{(2)} , g_{(2)}) \# h_{(3)} \cdot
g_{(3)} = f(g_{(1)} , h_{(1)}) \# g_{(2)} \cdot h_{(2)}
\tau(h_{(3)} , g_{(3)})
\end{eqnarray*}
Having in mind that $f$ is a coalgebra map we obatin, by applying
$\varepsilon \otimes Id$, that (SBR5) holds for $\tau$ and
therefore $\tau$ is a generalized $(u, v)$ - skew braiding.\\
We still need to prove that the compatibilities \equref{4.1} -
\equref{4.7} hold. Compatibilities \equref{4.1} - \equref{4.2} are
obtained from \equref{BR5} by considering: $a = 1$ and $g = 1$
respectively $b = 1$ and $h = 1$ while \equref{4.3} can be derived
from \equref{BR5} by considering $a = b = 1$ and then applying $Id
\otimes \varepsilon$. The next two compatibilities, \equref{4.4}
and \equref{4.5}, can be obtained by considering $h = t = 1$ and
$b = 1$ respectively $a = b = 1$ and $t = 1$ in \equref{BR3}. To
this end, relations \equref{4.6} and \equref{4.7} can be derived
from \equref{BR1} by considering $g = t = 1$ and $a = 1$
respectively $a = c = 1$ and $g = 1$.

Assume now that $(A, p)$ is a coquasitriangular Hopf algebra, $u$
is a generalized $(p,f)$ - right skew pairing, $v$ is a
generalized $(p,f)$ - left skew pairing, $\tau$ is a generalized
$(u, v)$ - skew braiding and $\sigma$ is given by \equref{39} such
that compatibilities \equref{4.1} - \equref{4.7} are fulfilled.
Then, using relations (RS2), (SBR2), (LS2) and the fact that $p$
is a coquasitriangular structure we can prove that for all $a \in
A$, $h \in H$ we have:
\begin{eqnarray*}
\sigma(1 \# 1 , a \# h) &{=}& u(1, h_{(1)}) p(1 , a_{(1)}) \tau(1
, h_{(2)}) v(1 , a_{(2)})\\
&{=}& \varepsilon(a)
\varepsilon(h)\\
&{=}& \varepsilon(a \# h)
\end{eqnarray*}
Moreover, using relations (RS4), (SBR4), (LS4) and again the fact
that $p$ is a coquasitriangular structure, we also have:
\begin{eqnarray*}
\sigma(a \# h, 1 \# 1) &{=}& u(a_{(1)} , 1) p(a_{(2)} , 1)
\tau(h_{(1)} , 1) v(h_{(2)} , 1)\\
&{=}& \varepsilon(a)
\varepsilon(h)\\
&{=}& \varepsilon(a \# h)
\end{eqnarray*}
for all $a \in A$, $h \in H$. Hence $\sigma$ also fulfills (BR4).

To prove that $\sigma$ satisfies (BR1) we start by first computing
the left hand side. Thus for all  $a, b, c \in A$ and $h, g, t \in
H$ we have:

\begin{eqnarray*}
LHS &{=}& u\bigl(a_{(1)}(g_{(1)} \triangleright b_{(1)}) f(g_{(3)}
\triangleleft b_{(3)} , h_{(1)}) , t_{(1)}\bigl) v\bigl((g_{(6)}
\triangleleft b_{(6)}) \cdot h_{(4)}, c_{(2)}\bigl) \\
&& p \bigl(a_{(2)}(g_{(2)}
\triangleright b_{(2)}) f(g_{(4)} \triangleleft b_{(4)} , h_{(2)}) , c_{(1)}\bigl)\tau\bigl((g_{(5)} \triangleleft b_{(5)}) \cdot h_{(3)} , t_{(2)}\bigl)\\
&\stackrel{(RS1)}{=}& u(a_{(1)} , t_{(1)}) u(g_{(1)}
\triangleright b_{(1)} , t_{(2)}) u\bigl(f(g_{(3)} \triangleleft
b_{(3)} , h_{(1)}) , t_{(3)}\bigl)v \bigl((g_{(6)} \triangleleft
b_{(6)}) \cdot h_{(4)}, c_{(4)}\bigl)\\&& p(a_{(2)} , c_{(1)})
p\bigl(\underline{f(g_{(4)} \triangleleft b_{(4)} , h_{(2)})} ,
c_{(3)}\bigl) \tau \bigl(\underline{(g_{(5)} \triangleleft
b_{(5)}) \cdot h_{(3)}} , t_{(4)}\bigl) p(g_{(2)}
\triangleright b_{(2)} , c_{(2)})\\
&\stackrel{(BE7)}{=}& u(a_{(1)} , t_{(1)}) u(g_{(1)}
\triangleright b_{(1)} , t_{(2)}) u\bigl(f(g_{(3)} \triangleleft
b_{(3)} , h_{(1)}) , t_{(3)}\bigl)\underline{v \bigl((g_{(6)}
\triangleleft b_{(6)}) \cdot h_{(4)}, c_{(4)}\bigl)}\\&& p(a_{(2)}
, c_{(1)}) \underline{p\bigl(f(g_{(5)} \triangleleft b_{(5)} ,
h_{(3)}) , c_{(3)}\bigl)} \tau \bigl((g_{(4)} \triangleleft
b_{(4)}) \cdot
h_{(2)} , t_{(4)}\bigl)p(g_{(2)} \triangleright b_{(2)} , c_{(2)})\\
&\stackrel{(LS1)}{=}& u(a_{(1)} , t_{(1)}) u(g_{(1)}
\triangleright b_{(1)} , t_{(2)}) \underline{u\bigl(f(g_{(3)}
\triangleleft b_{(3)} , h_{(1)}) , t_{(3)}\bigl)} p(a_{(2)} ,
c_{(1)})\\&& p(g_{(2)} \triangleright b_{(2)} , c_{(2)})
\underline{\tau \bigl((g_{(4)} \triangleleft b_{(4)}) \cdot
h_{(2)} , t_{(4)}\bigl)} v(g_{(5)} \triangleleft b_{(5)} ,
c_{(3)})
v(h_{(3)} , c_{(4)})\\
&\stackrel{(SBR1)}{=}& u(a_{(1)} , t_{(1)}) u(g_{(1)}
\triangleright b_{(1)} , t_{(2)}) p(a_{(2)} , c_{(1)}) p(g_{(2)}
\triangleright b_{(2)} , c_{(2)})\\
&& \tau(g_{(3)} \triangleleft b_{(3)} , t_{(3)}) \tau(h_{(1)} ,
t_{(4)}) v(g_{(4)} \triangleleft b_{(4)} ,
c_{(3)}) v(h_{(2)} , c_{(4)})\\
&\stackrel{(BE6)}{=}& u(a_{(1)} , t_{(1)}) u(g_{(1)}
\triangleright b_{(1)} , t_{(2)}) p(a_{(2)} , c_{(1)})
\underline{p(g_{(3)}
\triangleright b_{(3)} , c_{(2)})}\\
&& \tau(g_{(2)} \triangleleft b_{(2)} , t_{(3)}) \tau(h_{(1)} ,
t_{(4)}) \underline{v(g_{(4)} \triangleleft b_{(4)} ,
c_{(3)})} v(h_{(2)} , c_{(4)})\\
&\stackrel{\equref{4.6}}{=}& u(a_{(1)} , t_{(1)})
\underline{u(g_{(1)} \triangleright b_{(1)} , t_{(2)})} p(a_{(2)}
, c_{(1)})
\underline{\tau(g_{(2)} \triangleleft b_{(2)} , t_{(3)})} \tau(h_{(1)} , t_{(4)})\\
&& v(g_{(3)} , c_{(2)}) p(b_{(3)} , c_{(3)}) v(h_{(2)} , c_{(4)})\\
&\stackrel{\equref{4.7}}{=}& u(a_{(1)} , t_{(1)})
 p(a_{(2)}
, c_{(1)}) \tau(g_{(1)} , t_{(2)}) u(b_{(1)} , t_{(3)}) \tau(h_{(1)} , t_{(4)})\\
&& v(g_{(2)} , c_{(2)}) p(b_{(2)} , c_{(3)}) v(h_{(2)} , c_{(4)})\\
&{=}& RHS
\end{eqnarray*}
where in the second equality we also used the fact that $p$ is a
coquasitriangular structure. To prove (BR3) we start again by
computing the left hand side. Thus for all $a, b, c \in A$ and $h,
g, t \in H$ we have:
\begin{eqnarray*}
LHS &{=}& u\bigl((a_{(1)} , (g_{(5)} \triangleleft c_{(5)}) \cdot
t_{(3)}\bigl) \underline{p\bigl(a_{(2)} , b_{(1)}(g_{(1)}
\triangleright c_{(1)}) f(g_{(3)} \triangleleft c_{(3)} ,
t_{(1)})\bigl)}\\&& \tau\bigl(h_{(1)} , (g_{(6)} \triangleleft
c_{(6)}) \cdot t_{(4)}\bigl)\underline{v\bigl(h_{(2)} ,
b_{(2)}(g_{(2)} \triangleright
c_{(2)}) f(g_{(4)} \triangleleft c_{(4)} , t_{(2)})\bigl)}\\
&\stackrel{(RS3)}{=}& u\bigl(a_{(1)} , \underline{(g_{(5)}
\triangleleft c_{(5)}) \cdot t_{(3)}}\bigl) p\bigl(a_{(2)} ,
f(g_{(3)} \triangleleft c_{(3)} , t_{(1)})\bigl) p(a_{(3)} ,
g_{(1)}
\triangleright c_{(1)}) p(a_{(4)} , b_{(1)})\\
&& \tau\bigl(h_{(1)} , (g_{(6)} \triangleleft c_{(6)}) \cdot
t_{(4)}\bigl)v\bigl(h_{(2)} , \underline{f(g_{(4)} \triangleleft
c_{(4)} , t_{(2)})}\bigl) v(h_{(3)} , g_{(2)} \triangleright
c_{(2)}) v(h_{(4)} , b_{(2)})\\
&\stackrel{(BE7)}{=}&
u\bigl(a_{(1)} , (g_{(4)} \triangleleft c_{(4)}) \cdot
t_{(2)}\bigl) p\bigl(a_{(2)} , f(g_{(3)} \triangleleft c_{(3)} ,
t_{(1)})\bigl) p(a_{(3)} , g_{(1)}
\triangleright c_{(1)}) p(a_{(4)} , b_{(1)})\\
&& \underline{\tau\bigl(h_{(1)} , (g_{(6)} \triangleleft c_{(6)})
\cdot t_{(4)}\bigl)v\bigl(h_{(2)} , f(g_{(5)} \triangleleft
c_{(5)} , t_{(3)})\bigl)} v(h_{(3)} , g_{(2)} \triangleright
c_{(2)}) v(h_{(4)} , b_{(2)})
\end{eqnarray*}
\begin{eqnarray*}
&\stackrel{(SBR3)}{=}& \underline{u\bigl(a_{(1)} , (g_{(4)}
\triangleleft c_{(4)}) \cdot t_{(2)}\bigl) p\bigl(a_{(2)} ,
f(g_{(3)} \triangleleft c_{(3)} , t_{(1)})\bigl)} p(a_{(3)} ,
g_{(1)}
\triangleright c_{(1)}) p(a_{(4)} , b_{(1)})\\
&& \tau(h_{(1)} , t_{(3)})\tau(h_{(2)} , g_{(5)} \triangleleft
c_{(5)}) v(h_{(3)} , g_{(2)} \triangleright c_{(2)}) v(h_{(4)} ,
b_{(2)})\\
&\stackrel{(RS3)}{=}& u(a_{(1)} , t_{(1)}) u(a_{(2)} ,
\underline{g_{(3)} \triangleleft c_{(3)}}) p(a_{(3)} , g_{(1)}
\triangleright c_{(1)}) p(a_{(4)} , b_{(1)})
\tau(h_{(1)} , t_{(2)})\\
&&\tau(h_{(2)} , g_{(4)} \triangleleft c_{(4)}) v(h_{(3)} ,
\underline{g_{(2)}
\triangleright c_{(2)}}) v(h_{(4)} , b_{(2)})\\
&\stackrel{(BE6)}{=}& u(a_{(1)} , t_{(1)}) \underline{u(a_{(2)} ,
g_{(2)} \triangleleft c_{(2)}) p(a_{(3)} , g_{(1)} \triangleright
c_{(1)})} p(a_{(4)} , b_{(1)})
\tau(h_{(1)} , t_{(2)})\\
&&\tau(h_{(2)} , g_{(4)} \triangleleft c_{(4)}) v(h_{(3)} ,
g_{(3)}
\triangleright c_{(3)}) v(h_{(4)} , b_{(2)})\\
&\stackrel{\equref{4.4}}{=}& u(a_{(1)} , t_{(1)})p(a_{(2)} ,
c_{(1)})u(a_{(3)} , g_{(1)}) p(a_{(4)} , b_{(1)})
\tau(h_{(1)} , t_{(2)})\\
&&\underline{\tau(h_{(2)} , g_{(3)} \triangleleft c_{(3)})
v(h_{(3)} , g_{(2)}
\triangleright c_{(2)})} v(h_{(4)} , b_{(2)})\\
&\stackrel{\equref{4.5}}{=}& u(a_{(1)} , t_{(1)})p(a_{(2)} ,
c_{(1)})u(a_{(3)} , g_{(1)}) p(a_{(4)} , b_{(1)})
\tau(h_{(1)} , t_{(2)})\\
&&v(h_{(2)} , c_{(2)})\tau(h_{(3)} , g_{(2)}) v(h_{(4)} , b_{(2)})\\
&{=}& RHS
\end{eqnarray*}
Note that in the second equality we used the fact that $p$ is a
coquasitriangular structure. In order to show that $\sigma$ also
fulfills (BR5) we need the following compatibilities that can be
easily derived from \equref{4.1} and \equref{4.2} by applying
$\varepsilon \otimes Id$:
\begin{eqnarray}
\eqlabel{1.2} v(h_{(1)} , b_{(1)})(h_{(2)} \triangleleft b_{(2)}) &{=}& h_{(1)}v(h_{(2)} , b)\\
\eqlabel{2.2} (g_{(1)} \triangleleft a_{(1)}) u(a_{(2)} , g_{(2)})
&{=}& u(a , g_{(1)}) g_{(2)}
\end{eqnarray}
Computing the left hand side of (BR5) we obtain:
\begin{eqnarray*}
LHS&{=}&u(a_{(1)} , g_{(1)})p(a_{(2)} , b_{(1)}) \tau(h_{(1)} ,
g_{(2)}) v(h_{(2)} , b_{(2)})a_{(3)}(h_{(3)} \triangleright
b_{(3)})\\&& \underline{f(h_{(4)} \triangleleft b_{(4)} , g_{(3)})} \# \underline{(h_{(5)} \triangleleft b_{(5)}) \cdot g_{(4)}}\\
&\stackrel{(BE7)}{=}& u(a_{(1)} , g_{(1)})p(a_{(2)} , b_{(1)})
\tau(h_{(1)} , g_{(2)}) \underline{v(h_{(2)} ,
b_{(2)})}a_{(3)}\underline{(h_{(3)} \triangleright
b_{(3)})}\\&& f(h_{(5)} \triangleleft b_{(5)}, g_{(4)}) \# \underline{(h_{(4)} \triangleleft b_{(4)})} \cdot g_{(3)}\\
&\stackrel{\equref{4.1}}{=}& u(a_{(1)} ,
g_{(1)})\underline{p(a_{(2)} ,
b_{(1)})} \tau(h_{(1)} , g_{(2)})\underline{a_{(3)}b_{(2)}}f(h_{(4)} \triangleleft b_{(4)} , g_{(4)})\# \\&&  h_{(2)} \cdot g_{(3)}v(h_{(3)} , b_{(3)})\\
&{=}& u(a_{(1)} , g_{(1)})\underline{\tau(h_{(1)} , g_{(2)})}b_{(1)}a_{(2)}p(a_{(3)} , b_{(2)})f(h_{(4)} \triangleleft b_{(4)} , g_{(4)})\#\\&& \underline{h_{(2)} \cdot g_{(3)}}v(h_{(3)} , b_{(3)})\\
&\stackrel{(SBR5)}{=}& \underline{u(a_{(1)} ,
g_{(1)})}b_{(1)}\underline{a_{(2)}}p(a_{(3)} , b_{(2)}) f(h_{(4)}
\triangleleft b_{(4)}, g_{(4)})\# \underline{(g_{(2)}} \cdot
h_{(1)})\\&& \tau(h_{(2)} , g_{(3)})v(h_{(3)} , b_{(3)})
\end{eqnarray*}
\begin{eqnarray*}
&\stackrel{\equref{4.2}}{=}& b_{(1)}(g_{(1)} \triangleright
a_{(1)})p(a_{(4)} , b_{(2)})
f(h_{(4)} \triangleleft b_{(4)} , g_{(5)})\#\underline{(g_{(2)} \triangleleft a_{(2)})} \cdot h_{(1)}\\&& \underline{u(a_{(3)} , g_{(3)})}\tau(h_{(2)} , g_{(4)})v(h_{(3)} , b_{(3)})\\
&\stackrel{\equref{2.2}}{=}& b_{(1)}(g_{(1)} \triangleright
a_{(1)})p(a_{(3)} , b_{(2)})
f(\underline{h_{(4)} \triangleleft b_{(4)}}, g_{(5)})\# (g_{(3)} \cdot h_{(1)})u(a_{(2)} , g_{(2)})\\&& \tau(h_{(2)} , g_{(4)})\underline{v(h_{(3)} , b_{(3)})}\\
&\stackrel{\equref{1.2}}{=}& b_{(1)}(g_{(1)} \triangleright
a_{(1)})p(a_{(3)} , b_{(2)}) \underline{f(h_{(3)} , g_{(5)})}\#
(g_{(3)} \cdot h_{(1)})u(a_{(2)} , g_{(2)})\\&&
\underline{\tau(h_{(2)} , g_{(4)})} v(h_{(4)} , b_{(3)})\\
&\stackrel{\equref{4.3}}{=}& b_{(1)}(g_{(1)} \triangleright
a_{(1)})p(a_{(3)} , b_{(2)})v(h_{(4)} , b_{(3)})\tau(h_{(3)} ,
g_{(5)})\underline{f(g_{(4)} , h_{(2)})}\# \\
&&(\underline{g_{(3)} \cdot h_{(1)}})u(a_{(2)} , g_{(2)})\\
&\stackrel{(BE7)}{=}& b_{(1)}(g_{(1)} \triangleright
a_{(1)})p(a_{(3)} , b_{(2)})v(h_{(4)} , b_{(3)})\tau(h_{(3)} ,
g_{(5)})f(\underline{g_{(3)}} , h_{(1)})\# \\
&& (g_{(4)} \cdot h_{(2)})\underline{u(a_{(2)} , g_{(2)})}\\
&\stackrel{\equref{2.2}}{=}& b_{(1)}(g_{(1)} \triangleright
a_{(1)})p(a_{(4)} , b_{(2)})v(h_{(4)} , b_{(3)})\tau(h_{(3)} ,
g_{(5)})f(g_{(2)} \triangleleft a_{(2)} , h_{(1)})\# \\
&&(\underline{g_{(4)}} \cdot h_{(2)})\underline{u(a_{(3)} , g_{(3)})}\\
&\stackrel{\equref{2.2}}{=}& b_{(1)}(g_{(1)} \triangleright
a_{(1)})p(a_{(5)} , b_{(2)})v(h_{(4)} , b_{(3)})\tau(h_{(3)} ,
g_{(5)})f(g_{(2)} \triangleleft a_{(2)} , h_{(1)}) \#\\
&&  \bigl((g_{(3)} \triangleleft a_{(3)}) \cdot h_{(2)}\bigl) u(a_{(4)} , g_{(4})\\
&{=}&RHS
\end{eqnarray*}
In the forth equality we used the fact that $p$ is a
coquasitriangular structure. Thus (BR5) holds for $\sigma$ and
this ends the proof.
\end{proof}

The following result which characterizes the coquasitriangular
structures on a double cross product can be obtained from
\thref{4.5} by considering $f = \varepsilon \ot \varepsilon$ to be
the trivial cocycle.

\begin{corollary}\colabel{2.7}
Let $A \bowtie H$ be a double cross product of Hopf algebras.
There is a bijective correspondence between:

$(i)$ The set of all coquasitriangular structures $\sigma$ on the
double cross product $A \bowtie H$;

$(ii)$ The set of all quadruples $(p, \tau, u, v)$, where $p: A
\otimes A \rightarrow k$, $\tau: H \otimes H \rightarrow k$, $u: A
\otimes H \rightarrow k$, $v: H \otimes A \rightarrow k$ are
linear maps such that $(A, p)$ and $(H, \tau)$ are
coquasitriangular Hopf algebras, $u$ and $v$ are skew pairings on
$(A, H)$ respectively on $(H, A)$ and the following
compatibilities are fulfilled:
\begin{eqnarray*}
v(h_{(1)} , b_{(1)}) (h_{(2)} \triangleright b_{(2)}) \otimes
(h_{(3)} \triangleleft b_{(3)}) &{=}& b_{(1)} \otimes
h_{(1)}v(h_{(2)} ,
b_{(2)})\\
(g_{(1)} \triangleright a_{(1)}) \otimes (g_{(2)} \triangleleft
a_{(2)}) u(a_{(3)} , g_{(3)}) &{=}& u(a_{(1)} , g_{(1)})a_{(2)}
\otimes
g_{(2)}\\
u(a_{(1)} , g_{(2)} \triangleleft c_{(2)}) p(a_{(2)} , g_{(1)}
\triangleright c_{(1)}) &{=}& p(a_{(1)} , c) u(a_{(2)} , g)
\end{eqnarray*}
\begin{eqnarray*}
\tau(h_{(1)} , g_{(2)} \triangleleft c_{(2)}) v(h_{(2)} , g_{(1)}
\triangleright c_{(1)}) &{=}& v(h_{(1)} , c) \tau(h_{(2)} , g)\\
p(h_{(1)} \triangleright b_{(1)}, c_{(1)}) v(h_{(2)} \triangleleft
b_{(2)} ,
c_{(2)}) &{=}& v(h , c_{(1)}) p(b , c_{(2)})\\
u(h_{(1)} \triangleright b_{(1)}, t_{(1)}) \tau(h_{(2)}
\triangleleft b_{(2)} , t_{(2)}) &{=}& \tau(h , t_{(1)}) u(b ,
t_{(2)})
\end{eqnarray*}
Under the above bijection the coquasitriangular structure $\sigma:
(A \bowtie H) \otimes (A \bowtie H) \rightarrow k$ corresponding
to $(p, \tau, u, v)$ is given by:
\begin{equation}\eqlabel{23}
\sigma(a \ltimes h, b \ltimes g) = u(a_{(1)} , g_{(1)}) p(a_{(2)}
, b_{(1)}) \tau(h_{(1)} , g_{(2)}) v(h_{(2)} , b_{(2)})
\end{equation}
for all $a, b, c \in A$ and $h, g, t \in H$.
\end{corollary}

\section{Applications: coquasitriangular structures on generalized quantum doubles}

Let $A$ and $H$ be two Hopf algebras and $\lambda: H \ot A
\rightarrow k$ be a skew pairing. Then $(A, H)$ is a matched pair
of Hopf algebras with the following two actions:
\begin{eqnarray*}
h \triangleleft a &{=}& h_{(2)} \lambda^{-1}(h_{(1)} , a_{(1)})
\lambda(h_{(3)} , a_{(2)})\\
h \triangleright a &{=}& a_{(2)} \lambda^{-1}(h_{(1)} , a_{(1)})
\lambda(h_{(2)} , a_{(3)})
\end{eqnarray*}
The corresponding double cross product is called the
\textit{generalized quantum double} and it will be denoted by $A
\bowtie_{\lambda} H$ (\cite[Example 7.2.6]{majid}). As a special
case of \coref{2.7} we get:

\begin{theorem}\thlabel{3.1}
Let $A$ and $H$ be two Hopf algebras and $\lambda: H \ot A
\rightarrow k$ be a skew pairing. There is a bijective
correspondence between:

$(i)$ The set of all coquasitriangular structures $\sigma$ on the
generalized quantum double $A \bowtie _{\lambda} H$;

$(ii)$ The set of all quadruples $(p, \tau, u, v)$, where $p: A
\otimes A \rightarrow k$, $\tau: H \otimes H \rightarrow k$, $u: A
\otimes H \rightarrow k$, $v: H \otimes A \rightarrow k$ are
linear maps such that $(A, p)$ and $(H, \tau)$ are
coquasitriangular Hopf algebras, $u$ and $v$ are skew pairings on
$(A, H)$ respectively on $(H, A)$ and the following
compatibilities are fulfilled:
\begin{eqnarray}
\eqlabel{24} v(h_{(1)}, b_{(1)}) b_{(3)} \ot \lambda^{-1}(h_{(2)},
b_{(2)}) \lambda(h_{(4)}, b_{(4)}) h_{(3)} &{=}& b_{(1)} \ot
h_{(1)}
v(h_{(2)}, b_{(2)})\\
\eqlabel{25} a_{(2)} \lambda^{-1}(g_{(1)}, a_{(1)})
\lambda(g_{(3)}, a_{(3)}) \ot g_{(2)} u(a_{(4)}, g_{(4)}) &{=}&
u(a_{(1)}, g_{(1)}) a_{(2)}
\ot g_{(2)}\\
\eqlabel{26} u(a_{(1)}, g_{(2)}) \lambda(g_{(3)}, c_{(3)})p
(a_{(2)}, c_{(2)}) \lambda^{-1}(g_{(1)}, c_{(1)}) &{=}&
p(a_{(1)}, c)u(a_{(2)}, g)\\
\eqlabel{27}  \tau(h_{(1)}, g_{(2)}) \lambda(g_{(3)}, c_{(3)})
v(h_{(2)}, c_{(2)}) \lambda^{-1}(g_{(1)}, c_{(1)}) &{=}&
v(h_{(1)}, c)\tau(h_{(2)}, g)\\
\eqlabel{28} p(b_{(2)}, c_{(1)}) \lambda^{-1}(h_{(1)}, b_{(1)})
v(h_{(2)}, c_{(2)}) \lambda(h_{(3)}, b_{(3)}) &{=}& p(b, c_{(2)}) \lambda(h, c_{(1)})\\
\eqlabel{29} u(b_{(2)}, t_{(1)}) \lambda^{-1}(h_{(1)}, b_{(1)})
\tau(h_{(2)}, t_{(2)}) \lambda(h_{(3)}, b_{(3)}) &{=}& \tau(h,
t_{(1)}) u(b, t_{(2)})
\end{eqnarray}
Under this correspondence the coquasitriangular structure $\sigma:
(A \bowtie_{\lambda} H) \otimes (A \bowtie_{\lambda} H)
\rightarrow k$ corresponding to $(p, \tau, u, v)$ is given by:
\begin{equation}\eqlabel{30}
\sigma(a \ot h, b \ot g) = u(a_{(1)} , g_{(1)}) p(a_{(2)} ,
b_{(1)}) \tau(h_{(1)} , g_{(2)}) v(h_{(2)} , b_{(2)})
\end{equation}
for all $a, b, c \in A$ and $h, g, t \in H$.
\end{theorem}

\begin{theorem}\thlabel{3.2}
Let $(A, p)$ and $(H, \tau)$ be two coquasitriangular Hopf
algebras and $\lambda: H \ot A \rightarrow k$ be a skew pairing.
Then the generalized quantum double $A \bowtie_{\lambda} H$ is a
coquasitriangular Hopf algebra with the coquasitriangular
structure given by:
\begin{equation}
\sigma(a \bowtie h , b \bowtie g) = \lambda\bigl(S(g_{(1)} ,
a_{(1)})\bigl)p(a_{(2)} , b_{(1)}) \tau(h_{(1)} , g_{(2)})
\lambda(h_{(2)} , b_{(2)})
\end{equation}
\end{theorem}

\begin{proof}
We make use of \thref{3.1}: take $v:= \lambda$ and $u:=
\lambda^{-1} \circ \nu$, where $\nu$ is the flip map. We need to
prove that relations \equref{24} - \equref{29} are fulfilled. We
have:
\begin{eqnarray*}
LHS \equref{24} &{=}& b_{(3)} \ot \underline{\lambda(h_{(1)} ,
b_{(1)}) \lambda \bigl(S(h_{(2)}) , b_{(2)}\bigl)} \lambda(h_{(4)}, b_{(4)}) h_{(3)}\\
&{=}& b_{(1)} \ot h_{(1)}\lambda(h_{(2)} , b_{(2)}) = RHS
\equref{24}
\end{eqnarray*}

\begin{eqnarray*}
LHS \equref{25} &{=}& a_{(2)} \lambda\bigl(S(g_{(1)}) ,
a_{(1)}\bigl) \underline{\lambda(g_{(3)} , a_{(3)})
\lambda\bigl(S(g_{(4)}) , a_{(4)}\bigl)} \ot g_{(2)}\\
&{=}& a_{(2)} \lambda\bigl(S(g_{(1)}) , a_{(1)}\bigl) \ot g_{(2)}
= RHS \equref{25}
\end{eqnarray*}

\begin{eqnarray*}
LHS\equref{26} &{=}& \underline{\lambda\bigl(S(g_{(2)}) ,
a_{(1)}\bigl)} \lambda(g_{(3)} , c_{(3)}) p(a_{(2)} , c_{(2)})
\underline{\lambda
\bigl(S(g_{(1)}) , c_{(1)} \bigl)}\\
&{=}& \lambda\bigl(S(g_{(1)}) , \underline{c_{(1)}a_{(1)}}\bigl)
\underline{p(a_{(2)} , c_{(2)})}  \lambda(g_{(2)} , c_{(3)})\\
&{=}& \lambda\bigl(S(g_{(1)}) , a_{(2)}c_{(2)}\bigl) p(a_{(1)} ,
c_{(1)}) \lambda(g_{(2)} , c_{(3)})\\
&{=}& \underline{\lambda\bigl(S(g_{(2)}) , c_{(2)}\bigl)}
\lambda\bigl(S(g_{(1)}) , a_{(2)}\bigl) p(a_{(1)} , c_{(1)})
\underline{\lambda(g_{(3)} , c_{(3)})}\\
&{=}& \lambda\bigl(S(g_{(2)})g_{(3)} , c_{(2)}\bigl)
\lambda\bigl(S(g_{(1)}) , a_{(2)}\bigl) p(a_{(1)} , c_{(1)})\\
&{=}& \lambda\bigl(S(g) , a_{(2)}\bigl) p(a_{(1)} , c) =
RHS\equref{26}
\end{eqnarray*}
\begin{eqnarray*}
LHS\equref{27} &{=}& \tau(h_{(1)}, g_{(2)})
\underline{\lambda(g_{(3)}, c_{(3)}) \lambda(h_{(2)}, c_{(2)})}
\lambda \bigl(S(g_{(1)}), c_{(1)}\bigl)\\
&{=}&\underline{\tau(h_{(1)}, g_{(2)})}
\lambda(\underline{h_{(2)}g_{(3)}}, c_{(2)}) \lambda
\bigl(S(g_{(1)}), c_{(1)}\bigl)\\
&{=}& \tau(h_{(2)}, g_{(3)}) \lambda(g_{(2)}h_{(1)}, c_{(2)})
\lambda \bigl(S(g_{(1)}), c_{(1)}\bigl)\\
&{=}& \tau(h_{(2)}, g_{(3)}) \underline{\lambda(g_{(2)}, c_{(2)})}
\lambda(h_{(1)}, c_{(3)}) \underline{\lambda \bigl(S(g_{(1)}),
c_{(1)}\bigl)}\\
&{=}& \tau(h_{(2)}, g) \lambda(h_{(1)}, c) = RHS\equref{28}
\end{eqnarray*}
\begin{eqnarray*}
LHS\equref{28} &{=}& p(b_{(2)}, c_{(1)}) \lambda \bigl(S(h_{(1)}),
b_{(1)}\bigl) \underline{\lambda(h_{(2)},
c_{(2)}) \lambda(h_{(3)}, b_{(3)})}\\
&{=}& \underline{p(b_{(2)}, c_{(1)})} \lambda \bigl(S(h_{(1)}),
b_{(1)}\bigl) \lambda(h_{(2)}, \underline{b_{(3)}c_{(2)}})\\
&{=}& p(b_{(3)}, c_{(2)}) \lambda \bigl(S(h_{(1)}), b_{(1)}\bigl)
\lambda(h_{(2)}, c_{(1)}b_{(2)})\\
&{=}& p(b_{(3)}, c_{(2)}) \underline{\lambda \bigl(S(h_{(1)}),
b_{(1)}\bigl) \lambda(h_{(2)}, b_{(2)})} \lambda(h_{(3)},
c_{(1)})\\
&{=}& p(b, c_{(2)}) \lambda(h, c_{(1)}) = RHS\equref{28}
\end{eqnarray*}
\begin{eqnarray*}
LHS\equref{29} &{=}& \underline{\lambda \bigl(S(t_{(1)}),
b_{(2)}\bigl) \lambda\bigl(S(h_{(1)}), b_{(1)}\bigl)}
\tau(h_{(2)}, t_{(2)}) \lambda(h_{(3)}, b_{(3)})\\
&{=}& \lambda \bigl(S(\underline{t_{(1)}h_{(1)}}), b_{(1)}\bigl)
\underline{\tau(h_{(2)}, t_{(2)})} \lambda (h_{(3)}, b_{(2)})
\end{eqnarray*}
\begin{eqnarray*}
&{=}& \lambda \bigl(S(t_{(2)})S(h_{(2)}), b_{(1)}\bigl)
\tau(h_{(1)}, t_{(1)}) \lambda(h_{(3)}, b_{(2)})\\
&{=}& \lambda \bigl(S(t_{(2)}), b_{(1)}\bigl)
\underline{\lambda(S(h_{(2)}), b_{(2)})} \tau(h_{(1)}, t_{(1)})
\underline{\lambda(h_{(3)}, b_{(3)})}\\
&{=}& \tau(h, t_{(1)}) \lambda \bigl(S(t_{(2)}), b\bigl) =
RHS\equref{29}
\end{eqnarray*}
\end{proof}

As a consequence, we derive the necessary and sufficient
conditions for the generalized quantum double to be a
coquasitriangular Hopf algebra.

\begin{corollary}\colabel{3.3}
Let $A$ and $H$ be two Hopf algebras and $\tau: H \ot A
\rightarrow k$ be a skew pairing. Then the generalized quantum
double $A \bowtie_{\tau} H$ is a coquasitriangular Hopf algebra if
and only if both Hopf algebras $A$ and $H$ are coquasitriangular.
\end{corollary}

Also as a special case of \thref{3.2} we recover Majid's result
\cite[Proposition 7.3.1]{majid}:

\begin{corollary}
Let $(A, p)$ be a coquasitriangular Hopf algebra. Then the
generalized quantum double $A \bowtie_{p} A$ has a
coquasitriangular structure given by:
$$\sigma(a \ot b, c \ot d) = p\bigl(S(d_{(1)}), a_{(1)}\bigl) p(a_{(2)}, c_{(1)}) p(b_{(1)}, d_{(2)}) p(b_{(2)}, c_{(2)})$$
\end{corollary}
\begin{proof}
Consider $A = H$ and $\sigma = \tau := p$ in \thref{3.2}.
\end{proof}

\begin{examples}
$1)$ Consider the group algebra $k \ZZ$ with the obvious Hopf
algebra structure and let $g$ be a generator of $\ZZ$ in
multiplicative notation. We have a coquasitriangular structure $p:
k \ZZ \ot k \ZZ\rightarrow k$
given by: $p(g^{t}, g^{l}) = q^{tl}$. \\
Now consider the polynomial algebra $k[X]$ with the coalgebra
structure and the antipode given by:
$$\Delta(X^{n}) = \sum_{k=0}^{n} \binom{n}{k} X^{k} \ot X^{n-k}, \quad \varepsilon(X^{n}) = 0, \quad S(X^{n}) = (-1)^{n}X^{n}, \quad \mbox{for all}\quad n>0$$
Any element $\alpha \in k$, induces a coquasitriangular structure
$\tau$ on $k[X]$ as follows:
$$\tau(X^{i}, X^{j}) = \left \{\begin{array}{rcl}
0, \, & \mbox { if }& i \neq j\\
i!\alpha^{i}, \, & \mbox { if }& i=j
\end{array} \right.
$$
Moreover, there is a skew pairing $\lambda$ between the two Hopf
algebras $k[X]$ and $k \ZZ$ given by:
$$\lambda(X^{m}, g^{t}) = t^{m}$$
with the convention that $t^{0} = 1$ even if $t = 0$. Thus, by
applying \thref{3.2} we obtain a coquasitriangular structure
$\sigma$ on the generalized quantum double $k \ZZ
\bowtie_{\lambda} k[X]$:
$$\sigma(g^{t} \ot X^{n}, g^{l} \ot X^{m}) = \sum_{k \in \overline{0,m}, r \in \overline{0,n}}^{r+k=m} (-1)^{k} \binom{m}{k} \binom{n}{r} t^{k} q^{tl} l^{n-r} r! \alpha^{r}$$

$2)$ Let $k$ be a field with char$k \neq 2$ and $H_{4}$ be
Sweedler's Hopf algebra. That is, $H_{4}$ is generated as an
algebra by  elements $g$ and $x$ subject to relations:
$$g^{2} = 1, \quad x^{2}=0, \quad xg = -gx$$
The coalgebra structure and the antipode are given by:
$$\Delta(g) = g \ot g, \quad \Delta(x) = x \ot g + 1 \ot x, \quad \varepsilon(g) = 1, \quad \varepsilon(x) = 0$$
$$S(g) = g, \quad S(x) = gx$$
For any $\alpha \in k$ the map $p_{\alpha}:H_{4} \ot H_{4}
\rightarrow k$ is a coquasitriangular structure on $H_{4}$, where
$p_{\alpha}$ is defined as:
\begin{center}
\begin{tabular} {l | r  r  r  r  }
$p_{\alpha}$ & 1 & $g$ & $x$ & $gx$\\
\hline 1 & 1 & 1 & 0 & 0\\
$g$ & 1 & -1 & 0 & 0 \\
$x$ & 0 & 0 & $\alpha$ & $\alpha$\\
$gx$ & 0 & 0 & $\alpha$ & $\alpha$ \\
\end{tabular}
\end{center}
Let $\alpha$, $\beta$, $\gamma \in k$ and consider $p_{\alpha}$,
$p_{\beta}$, $p_{\gamma}$ the corresponding coquasitriangular
structures on $H_{4}$. Since any coquasitriangular structure is in
particular a skew pairing, we can construct the generalized
quantum double $H_{4} \bowtie_{p_{\gamma}} H_{4}$. In view of
\thref{3.2} there is a coquasitriangular structure on $H_{4}
\bowtie_{p_{\gamma}} H_{4}$ given by:
$$\sigma(a \ot h, b \ot g) = p_{\gamma}\bigl(S(g_{(1)}), a_{(1)}\bigl)p_{\alpha}(a_{(2)}, b_{(1)}) p_{\beta}(h_{(1)}, g_{(2)}) p_{\gamma}(h_{(2)}, b_{(2)})$$

$3)$ Let $k$ be a field with char$(k) \neq 2$ and consider the
$k$-algebra $\widetilde{U(n)}$ defined by generators $\{c, x_{1},
..., x_{n}, y_{1}, ..., y_{n}\}$ and relations:
\begin{eqnarray*}
c^{2} = 1, \quad x_{i}^{2} = y_{i}^{2} = 0, \quad cx_{i} + x_{i}c
= 0, \quad cy_{i} +  y_{i}c = 0\\
x_{i}x_{j} + x_{j}x_{i} = 0, \quad y_{i}y_{j} + y_{j}y_{i} = 0,
\quad x_{i}y_{j} = y_{j} x_{i}, \quad 1 \leq i \leq
n\\
\end{eqnarray*}
$\widetilde{U(n)}$ has a Hopf algebra structure given by:
\begin{eqnarray*}
\Delta(c) = c \ot c, \quad \Delta(x_{i}) = 1 \ot x_{i} + x_{i} \ot
c, \quad \Delta(y_{i}) = c \ot y_{i} + y_{i} \ot 1\\
\varepsilon(c) = 1, \quad \varepsilon(x_{i}) = \varepsilon(y_{i})
= 0, \quad S(c) = c, \quad S(x_{i}) = cx_{i}, \quad S(y_{i}) =
y_{i}c, \quad 1 \leq i \leq n
\end{eqnarray*}
$\widetilde{U(n)}$ is a quotient of the Hopf algebra $U(n)$
introduced by Takeuchi in \cite{Ta}. Now consider
$\widetilde{B_{-}}$ and $\widetilde{B_{+}}$ to be the Hopf
subalgebras of $\widetilde{U(n)}$ generated by $\{c, y_{1}, ...,
y_{n}\}$ respectively $\{c, x_{1}, ..., x_{n}\}$. These are the
so-called \textit{Borel subalgebras}. $\widetilde{B_{-}}$ and
$\widetilde{B_{+}}$ are coquasitriangular Hopf algebras with:
$$\tau: \widetilde{B_{-}} \ot \widetilde{B_{-}} \rightarrow k, \quad \tau(c, c) = -1, \quad \tau(c, y_{i}) = \tau(y_{j}, c) = 0, \quad \tau(y_{i}, y_{j}) = \alpha_{ij}, \quad 1\leq i, j \leq n$$
$$p: \widetilde{B_{+}} \ot \widetilde{B_{+}} \rightarrow k, \quad p(c, c) = -1, \quad p(c,x_{i}) = p(x_{j}, c) = 0, \quad p(x_{i}, x_{j}) = \beta_{ij}, \quad 1\leq i, j \leq n$$
Moreover, there is a skew-pairing $\lambda: \widetilde{B_{-}} \ot
\widetilde{B_{+}} \rightarrow k$ given by:
$$\lambda(c, c) = -1, \quad \lambda(c, x_{j}) = \lambda(y_{i}, c) = 0, \quad \lambda(y_{i}, x_{j}) = \delta_{i j}, \quad 1\leq i, j \leq n$$
where $\delta_{i,j}$ is the Kronecker delta. Therefore, using
\thref{3.2} the generalized quantum double $\widetilde{B_{+}}
\bowtie_{\tau} \widetilde{B_{-}}$ is a coquasitriangular Hopf
algebra with the coquasitriangular structure $\sigma:
(\widetilde{B_{+}} \bowtie_{\tau} \widetilde{B_{-})} \ot
(\widetilde{B_{+}} \bowtie_{\tau} \widetilde{B_{-}}) \rightarrow
k$ given by:

\begin{center}
\begin{tabular} {c | c  c  c  c  }
$\sigma$ & $c \bowtie c$ & $c \bowtie y_{i}$ & $x_{j} \bowtie c$ & $x_{k} \bowtie y_{l}$\\
\hline $c \bowtie c$ & 1 & 0 & 0 & 0\\
$c \bowtie y_{s}$ & 0 & $\alpha_{s i}$ & $\delta_{s j}$ & 0 \\
$x_{m} \bowtie c$ & 0 & $-\delta_{m j}$ & $\beta_{m j}$ & 0\\
$x_{n} \bowtie y_{r}$ & 0 & 0 & 0 & $\alpha_{r l} \beta_{n k} - \delta_{r k} \delta_{l n}$ \\
\end{tabular}
\end{center}
\end{examples}

\section{Acknowledgment}
The author is supported by an ''Aspirant'' Fellowship from the
Fund for Scientific Research--Flanders (Belgium) (F.W.O.
Vlaanderen). This research is part of the grant no. 88/05.10.2011
of the Romanian National Authority for Scientific Research,
CNCS-UEFISCDI.

\end{document}